\documentclass[11pt,reqno]{amsart}
\usepackage{amsmath,amssymb,amsthm}
\usepackage[hidelinks]{hyperref}
\newtheorem{thm}{Theorem}[section]

\theoremstyle{definition}

\newtheorem{definition}[thm]{Definition}

\newtheorem{lemma}[thm]{Lemma}

\newtheorem{theorem}[thm]{Theorem}

\numberwithin{equation}{section}

\usepackage{enumitem}

\def\N{\mathbb{N}}

\title{On the generalized Sierpi\'nski and Riesel numbers}
\author{Paulius Virbalas}
\address{Institute of Mathematics, Faculty of Mathematics and Informatics, Vilnius
University, Naugarduko 24, Vilnius LT-03225, Lithuania}
\email{paulius.virbalas@mif.vu.lt}


\begin{document}

\begin{abstract}
In this paper, we prove that for any fixed base $b \ge 2$, there exists an infinite arithmetic progression of positive integers $k$ that are simultaneously generalized Sierpi\'nski and generalized Riesel numbers in base $b$. More precisely, we construct such an arithmetic progression so that both $k\cdot b^n + 1$ and $k\cdot b^n - 1$ have at least two distinct prime factors for every positive integer $n$. Our approach uses covering systems, following the ideas of Erd\H{o}s and later constructions of Harrington. The argument is then completed using Zsigmondy's theorem on primitive divisors together with properties of multiplicative orders.
\end{abstract}
\maketitle

\section{Introduction}\label{intro}

In 1849, de Polignac suggested that every positive odd integer greater than $3$ can be expressed as a sum of an odd prime and a power of $2$. The counterexample $127$ was soon found by de Polignac himself (in fact, the counterexample $959$ was known to Euler many years earlier) \cite{b16}. Nevertheless, the question of whether there are infinitely many counterexamples was answered only in 1950 by Erd\H{o}s \cite{erd50}. In particular, by introducing a novel approach based on covering systems, Erd\H{o}s demonstrated that there exist infinitely many odd positive integers $k$ for which $|k-2^n|$ is composite for all positive integers $n$. From this it follows that any such $k$ cannot be written in the form $2^n+p$ for an odd prime $p$.\par
Recall that a covering system is a finite
collection of arithmetic progressions (expressed as congruences) whose union covers the set of all integers. For example, the collection of five congruences 
\begin{equation}\label{cs}
\left\{ a_i \pmod{m_i} \right\}_{i=1}^{5}
= \left\{
\begin{array}{l}
0 \pmod{2},\quad 0 \pmod{3}, \quad 1 \pmod{4}, \\
5 \pmod{6},\quad 7 \pmod{12}
\end{array}\right\}
\end{equation}
is a covering system, because it can be checked that every integer is congruent to $a_i \pmod{m_i}$ for at least one value of $i \in \{1, \ldots,5\}$. Since their introduction, covering systems have been the subject of extensive investigation in their own right (see, e.g., \cite{c1, c3, c2}). Combined with Zsigmondy's theorem on primitive divisors, covering systems have become the key tool in proving that numbers of a certain form can never be prime. Classical (and perhaps most well-known) examples from this line of research are centered around Riesel numbers and Sierpi\'nski numbers. In 1956, Riesel \cite{riesel56} showed that there exist infinitely many odd positive integers $k$ so that $k\cdot 2^n-1$ is composite for all positive integers $n$ (such $k$ became known as Riesel numbers). Similarly, in 1960, Sierpi\'nski \cite{sierp60} showed that there exist infinitely many odd positive integers $k$ so that $k\cdot 2^n+1$ is composite for all positive integers $n$ (such $k$ became known as Sierpi\'nski numbers). Since the original findings of Riesel and Sierpi\'nski, considerable interest has been devoted to investigating the existence of Riesel and Sierpi\'nski numbers within various integer sequences (see, e.g. \cite{sq1, sq2, sq4, sq3}). It is also worth noting that, despite numerous efforts and extensive computer calculations, determining the smallest Riesel and Sierpi\'nski numbers remains an open problem. As of 2026, the smallest known Riesel and Sierpi\'nski numbers are $509203$ and $78557$, respectively \cite{sr26}.\par
Several directions have been explored in generalizing the notion of Sierpi\'nski and Riesel numbers (see, e.g., \cite{gs3, gs1, gs2}). In this paper, we focus on the generalized Sierpi\'nski and Riesel numbers in base $b$, where $b$ is any integer satisfying $b\geq 2$. Following the work of Evans, Harrington, Heiney, and Wieczorek \cite{nn}, we use the following definitions. As usual, for positive integers $x$ and $y$, by $\operatorname{gcd}(x,y)$ we denote the greatest common divisor of $x$ and $y$.
\begin{definition}[Sierpi\'nski number in base $b$]\label{sierpn}
For a fixed integer $b\geq 2$, a positive integer $k$ is called a $b$-Sierpi\'nski number if
$\operatorname{gcd}(k+1,b-1)=1$ and $k\cdot b^n+1$ is composite for all positive integers $n$.
\end{definition}
\begin{definition}[Riesel number in base $b$]\label{riesn}
For a fixed integer $b\geq 2$, a positive integer $k$ is called a $b$-Riesel number if
$\operatorname{gcd}(k-1,b-1)=1$ and $k\cdot b^n-1$ is composite for all positive integers $n$.
\end{definition}
The requirement that $\operatorname{gcd}(k+\beta,b-1)=1$, with $\beta=1$ in the Sierpi\'nski case and $\beta=-1$ in the Riesel case, is imposed to exclude the trivial situation in which every term of $k\cdot b^n+\beta$ has a common divisor greater than $1$. We note that the notions of $2$-Sierpi\'nski and $2$-Riesel numbers defined above are not equivalent to the classical notions of Sierpi\'nski and Riesel numbers, respectively, since the former definitions do not require $k$ to be odd. For an alternative definition of generalized Sierpi\'nski numbers in base $b\geq2$, together with a discussion of the role of the oddness condition in the classical definition, we refer the reader to \cite{br09} (see also \cite{doc09}).\par
Brier was the first to study numbers that are simultaneously Sierpi\'nski and Riesel numbers. According to Filaseta, Finch, and Kozek \cite{fil08}, in an unpublished paper from 1998 Brier showed that there exist infinitely many such numbers. These numbers subsequently became known as Brier numbers and were studied further by other authors (see, e.g., \cite{cg07, fil08}). In view of Definitions~\ref{sierpn} and \ref{riesn}, the following definition is natural (and has already been used in \cite{nn}).
\begin{definition}[Brier number in base $b$]
For a fixed integer $b\geq 2$, a positive integer $k$ is called a $b$-Brier number if it is simultaneously a $b$-Sierpi\'nski number and a $b$-Riesel number.
\end{definition}
Recently, generalized Sierpi\'nski and Riesel numbers in Cullen-type sequences were studied in \cite{nn}. It follows from the results therein that there exist infinitely many $b$-Brier numbers for every base $b$ such that $b+1$ is not a power of $2$. In this paper, we prove the following result.
\begin{theorem}\label{maint}
For every fixed integer $b\geq 2$, there exists an infinite arithmetic progression of $b$-Brier numbers.
\end{theorem}
We prove Theorem~\ref{maint} by constructing an infinite arithmetic progression of positive integers $k$ satisfying the required coprimality conditions and such that both $k\cdot b^n+1$ and $k\cdot b^n-1$ have at least two distinct prime factors for every positive integer $n$. The necessary auxiliary results are collected in Section~\ref{aux}, while the proof of Theorem~\ref{maint} is given in Section~\ref{prf}.


\section{\textbf{Auxiliary results}}\label{aux}

A key ingredient in the proof of Theorem~\ref{maint} is the paper \cite{harr15} of Harrington, in which he constructs a covering system $\mathcal{C}$ such that all moduli in $\mathcal{C}$ are distinct and every integer satisfies at least three congruences in $\mathcal{C}$. In particular, $\mathcal{C}$ is the union of three covering systems $\mathcal{C}_1$, $\mathcal{C}_2$, and $\mathcal{C}_3$, depending on four suitably chosen distinct primes $\pi_1,\pi_2,\pi_3,\pi_4$. For our purposes, it is important to note that for any choice of $\pi_1,\pi_2,\pi_3,\pi_4$, the system $\mathcal{C}_1$ contains the modulus $2$, and hence all moduli occurring in $\mathcal{C}_2$ and $\mathcal{C}_3$ are greater than $2$. For the remainder of the paper, we fix Harrington's choice
\[
\pi_1=11,\qquad \pi_2=13,\qquad \pi_3=17,\qquad \pi_4=19,
\]
and for these primes we set
\begin{equation*}
\mathcal{S}_1:=\mathcal{C}_2,\qquad \mathcal{S}_2:=\mathcal{C}_3.
\end{equation*}
From the construction described in Section~4 of \cite{harr15}, we obtain
\[\mathcal{S}_1=\{a_i \pmod{m_i} \}_{i=1}^{C}=\left\{
\begin{array}{l}
1 \pmod{3},\quad 5 \pmod{6}, \quad 3 \pmod{9}, \\
2 \pmod{12}, \quad \ldots,\quad\quad \quad a_C \pmod{m_C}
\end{array}
\right\},\]
where the moduli satisfy
\[m_1<m_2<\ldots<m_C,\quad m_1=3, ~m_2=6,\]
and
\[\mathcal{S}_2=\{c_j \pmod{n_j} \}_{j=1}^{D}=\left\{
\begin{array}{l}
1 \pmod{5},\quad 1 \pmod{7}, \quad 7 \pmod{10}, \\
9 \pmod{14}, \quad \ldots,\quad\quad \quad c_D \pmod{n_D}
\end{array}
\right\},\]
where the moduli satisfy
\[n_1<n_2<\ldots<n_D,\quad n_1=5, ~n_2=7.\]
The exact numbers of moduli are not stated in the original source. According to our calculations, $C=277$ and $D=96314$, though these exact values will not be relevant to our proof. We will use only the facts that $\mathcal{S}_1$ and $\mathcal{S}_2$ are covering systems, that all their moduli are greater than $2$, that the moduli within each system are distinct, and that the two systems have no modulus in common.\par
The following auxiliary results are from elementary number theory.
\begin{lemma}[Zsigmondy's theorem, \cite{zsg}]\label{zsig}
Let $x > y$ be coprime positive integers. Then for any integer $m > 1$, there exists a prime number $p$ (called a primitive prime divisor) such that $x^m-y^m$ is divisible by $p$, but $x^k-y^k$ is not divisible by $p$
for all integers $1 \leq k < m$,  
except in the following cases:
\begin{enumerate}[label=\textup{(\roman*)}]
    \item $x = 2, y = 1, m = 6$, in which case $2^6 - 1 = 63 = 7 \cdot 9$, and neither $3$ nor $7$ is a primitive divisor;
    \item $m = 2$ and $x + y$ is a power of 2.
\end{enumerate}
\end{lemma}

Let $x$ and $y$ be coprime positive integers with $y\geq2$. Recall that the order of $x$ modulo $y$ is the least positive integer $s$ such that $x^s \equiv 1 \pmod y$. We will use the notation $\operatorname{ord}_y(x)=s$.

\begin{lemma}\label{zsiga}
Let $b\geq3$ be a fixed integer. Then for any integer $m \geq 3$, there exists a prime number $p$ such that $m$ is the order of $b$ modulo $p$, i.e., $\operatorname{ord}_p(b)=m$.
\end{lemma}
\begin{proof}
We invoke Zsigmondy's theorem (Lemma~\ref{zsig}) with $x=b$ and $y=1$. Since $b\geq 3$ and $m\geq3$, the exceptional cases of Zsigmondy's theorem do not apply. Hence, $b^m-1$ has a primitive prime divisor for every $m \geq3$. The claim of the lemma follows from the definition of a primitive prime
divisor.
\end{proof}

For a positive integer $x$ and a prime number $p$, let $\nu_p(x)$ denote the $p$-adic valuation of $x$. That is, $\nu_p(x)$ equals the exponent of the highest power of $p$ that divides $x$. The following lemma extends Lemma~2 of \cite{m1} from base $2$ to an arbitrary base $b\geq2$ and from a single lifting step to all higher powers of $p$.

\begin{lemma}\label{ord}
Let $b\geq 2$ and $u \geq 1$ be fixed integers, and $p$ an odd prime. Suppose that $\operatorname{ord}_{p^{u}}(b)=m$, and that $\nu_p(b^m-1)=u$. Then for any non-negative integer $r$  both of the following hold:
\begin{enumerate}[label=\textup{(\roman*)}]
\item $\operatorname{ord}_{p^{u+r}}(b)=m\cdot p^r$;
\item $\nu_p(b^{m\cdot p^{r}}-1)=u+r$. 
\end{enumerate}
\end{lemma}
\begin{proof}
We use mathematical induction on $r$. The base case $r=0$ is true by assumption. Assume that (i) and (ii) hold for some non-negative integer $r=r'$. Then
\[b^{m\cdot p^{r'}}-1=p^{u+r'} \cdot l_0\]
with $l_0$ not divisible by $p$. We need to show that (i) and (ii) hold for $r=r'+1$. By applying the binomial expansion, we get
\begin{equation*}
\begin{split}
b^{m\cdot p^{r'+1}}&=(b^{m\cdot p^{r'}})^{p}\\
&=(1+p^{u+r'}\cdot l_0)^{p}\\
&=1+p^{u+r'+1}\cdot l_0+\binom{p}{2}\cdot (p^{u+r'}\cdot l_0)^2+\ldots+\binom{p}{p}\cdot (p^{u+r'}\cdot l_0)^{p}.
\end{split}
\end{equation*}
Therefore
\begin{equation}\label{lm}
b^{m\cdot p^{r'+1}}-1=p^{u+(r'+1)}\cdot l_1
\end{equation}
with $l_1$ clearly not divisible by $p$. Hence 
\[\nu_p(b^{m\cdot p^{r'+1}}-1)=u+(r'+1),\]
which completes the proof of (ii).\par
Let $w$ denote the order of $b$ modulo $p^{u+(r'+1)}$. Since by assumption, $m$ is the order of $b$ modulo $p^u$, it follows that $w=mk$ for some positive integer $k$. Also, by induction hypothesis we have that $\nu_p(b^{m\cdot p^{r'}}-1)=u+r'$. This implies that $b^{m\cdot p^{r'}}-1$ is not divisible by $p^{u+(r'+1)}$. Thus, $w=mk$ does not divide $m\cdot p^{r'}$. On the other hand, from \eqref{lm} we derive that $w=mk$ divides $m\cdot p^{r'+1}$. Such divisibility relations are possible if and only if $k=p^{r'+1}$, that is, $w=m\cdot p^{r'+1}$. This completes the proof of (i). 
\end{proof}

\section{\textbf{Proof of Theorem~\ref{maint}}}\label{prf}
The proof is based on techniques similar to those developed by Chen \cite{m4,m3,m1} and, more recently, by Zhao \cite{m5}.
\begin{proof} 
Let $b\geq 3$ be a fixed integer (the case $b=2$ will be treated separately at the end). Consider the two covering systems 
\[
\mathcal{S}_1=\{a_i \pmod {m_i}\}_{i=1}^{C}
\quad\text{and}\quad
\mathcal{S}_2=\{c_j \pmod{n_j} \}_{j=1}^{D}
\]
that were described in the first part of Section~\ref{aux}. By applying Lemma~\ref{zsiga},
we associate to each $m_i$ a prime $p_i$, and to each $n_j$ a prime $q_j$, such that
\[
m_i = \operatorname{ord}_{p_i}(b)
\quad\text{and}\quad
n_j = \operatorname{ord}_{q_j}(b).
\]
Since all the moduli in $\mathcal{S}_1$ and $\mathcal{S}_2$ are greater than $2$, and $b\geq 3$, neither of the exceptional cases in Lemma~\ref{zsig} occurs. Moreover, since the moduli within each covering system are distinct and the two systems have no modulus in common, all the numbers
\begin{equation}\label{listmn}
m_1, \ldots, m_{C}, n_1, \ldots, n_{D}
\end{equation}
are distinct. It follows that all the corresponding primes
\[p_1, \ldots, p_{C}, q_1, \ldots, q_{D}\]
are distinct as well. Set 
\[B:=\operatorname{max}\{m_1, \ldots, m_{C}, n_1, \ldots, n_{D}\}\]
and choose $N \in \N$  such that $2^N>B$. It is then clear that
\[p_i^N>B\quad\quad \text{and} \quad \quad q_j^N>B,\]
where $i=1, \ldots, C$ and $j=1, \ldots, D$.
By applying Lemma~\ref{zsiga} again, we associate to each $p_i^N$ a prime $p_i^{*}$, and to each $q_j^N$ a prime $q_j^{*}$, such that 
\begin{equation}\label{hord}
p_i^N = \operatorname{ord}_{p_i^{*}}(b)\quad\text{and}\quad q_j^{N} = \operatorname{ord}_{q_j^{*}}(b).
\end{equation}
The choice of $N$ guarantees that all primes 
\begin{equation}\label{primeslist}
p_1, \ldots, p_{C}, q_1, \ldots, q_{D}, p_1^{*}, \ldots, p_{C}^{*}, q_1^{*}, \ldots, q_{D}^{*}
\end{equation}
are distinct. Set
\begin{equation}\label{padic}
u_i:=\nu_{p_i}(b^{m_i}-1)\quad\quad \text{and} \quad \quad v_j:=\nu_{q_j}(b^{n_j}-1).
\end{equation}
Thus $b^{m_i}-1$  is divisible by $p_i^{u_i}$ but not by $p_i^{u_i+1}$, and analogously, $b^{n_j}-1$ is divisible by $q_j^{v_j}$ but not by $q_j^{v_j+1}$. Moreover, from the defining properties of the primes $p_i,q_j$ and the definitions of $u_i,v_j$, we have
\begin{equation}\label{pord}
\operatorname{ord}_{p_i^{u_i}}(b)=m_i\quad\text{and}\quad
\operatorname{ord}_{q_j^{v_j}}(b)=n_j.
\end{equation}
Also, choose $M$ such that 
\[b^M > \operatorname{max}\{p_i^{u_i+N}\cdot q_j^{v_j+N}~:~~i=1,\ldots, C~~\text{and}~~j=1,\ldots, D\}.\] 
Finally, by the Chinese remainder theorem, take any positive integer $k$ with $k > b^M$ satisfying all of the following:
\begin{equation}\label{modp}
k\cdot b^{a_i}\equiv -1 \pmod {p_i^{u_i+N} \cdot p_i^*},\quad i=1,\ldots, C;  
\end{equation}
\begin{equation}\label{modq}
k\cdot b^{c_j}\equiv 1 \pmod {q_j^{v_j+N}\cdot q_j^*},\quad j=1,\ldots, D;  
\end{equation}
\begin{equation}\label{modb}
k\equiv 1+b^M \pmod {b^{M+1}}. 
\end{equation}
We will demonstrate that both $k\cdot b^n+1$ and $k\cdot b^n-1$ have at least two distinct prime factors for every positive integer $n$. Fix any $n \in \N$. Since each of $\mathcal{S}_1$ and $\mathcal{S}_2$ is a covering system, the integer $n$ simultaneously satisfies
\[n \equiv a_i \pmod {m_i}\quad\quad\text{and}\quad\quad n \equiv c_j \pmod {n_j} \]
for some indices $i$ and $j$. Thus, there exist non-negative integers $s$ and $t$ such that 
\[n=m_i\cdot s + a_i \quad\quad\text{and}\quad\quad n=n_j\cdot t + c_j. \]
Consequently
\[k\cdot b^n +1  \equiv k\cdot b^{a_i} \cdot (b^{m_i})^s+1 \equiv k\cdot b^{a_i}\cdot 1 +1 \equiv -1+1 \equiv 0 \pmod {p_i}\]
and
\[k\cdot b^n -1  \equiv k\cdot b^{c_j} \cdot (b^{n_j})^t-1 \equiv k\cdot b^{c_j}\cdot1 - 1 \equiv 1-1 \equiv  0 \pmod {q_j}. \]
Hence
\[k\cdot b^n +1=p_i^{\alpha_i}\cdot P\quad\quad\text{and}\quad\quad k\cdot b^n -1=q_j^{\gamma_j}\cdot Q\]
for some positive integers $\alpha_i, \gamma_j, P, Q$, where $P$ is not divisible by $p_i$, and $Q$ is not divisible by $q_j$. To complete the proof, it is sufficient to demonstrate that $P$ is divisible by some prime $p \neq p_i$, and $Q$ is divisible by some prime $q \neq q_j$.\par
First, we will show that the claim is true for $P$. From
\eqref{modb} it follows that
\begin{equation}\label{sierp}
p_i^{\alpha_i}\cdot P =k\cdot b^n +1 \geq (1+b^M)\cdot b^1+1 > b^M > p_i^{u_i+N}.
\end{equation}
Hence, if $\alpha_i < u_i+N$, then clearly $P>1$ and we are done. In contrast, suppose that $\alpha_i \geq u_i+N$. It means that $k\cdot b^n + 1$ is divisible by $p_i^{u_i+N}$. This, together with \eqref{modp}, forces 
\[k\cdot b^n + 1 = k\cdot b^{a_i} \cdot (b^{m_i})^s+1 \equiv -1\cdot b^{m_i\cdot s} +1  \equiv 0 \pmod {p_i^{u_i+N}},\]
that is
\[b^{m_i\cdot s} \equiv 1 \pmod {p_i^{u_i+N}}. \]
By \eqref{padic} and \eqref{pord}, the assumptions of Lemma~\ref{ord} are satisfied. Hence,
$m_i\cdot p_i^N$ is the order of $b$ modulo $p_i^{u_i+N}$. Thus, $m_i\cdot s$ is divisible by $m_i\cdot p_i^N$, and hence $s$ is divisible by $p_i^N$. Write $s=p_i^N \cdot s'$ for some non-negative integer $s'$. Recall from \eqref{hord} that $p_i^N$ is the order of $b$ modulo $p_i^{*}$. Therefore 
\begin{equation}\label{P2}
b^{m_i\cdot s} \equiv (b^{p_i^N})^{m_i\cdot s'} \equiv 1^{m_i\cdot s'} \equiv 1 \pmod {p_i^{*}}.
\end{equation}
Combining \eqref{modp} with \eqref{P2} yields 
\[k\cdot b^n + 1 = k\cdot b^{a_i} \cdot b^{m_i\cdot s}+1 \equiv -1\cdot 1 +1  \equiv 0 \pmod {p_i^{*}}.\]
Hence, $P$ is divisible by the prime $p_i^{*} \neq p_i$. As a result, $k\cdot b^n + 1$ has at least two distinct prime factors, namely $p_i$ and $p_i^{*}$. \par
The reasoning with respect to $Q$ is analogous. From \eqref{modb} we obtain
\begin{equation*}
q_j^{\gamma_j}\cdot Q=k\cdot b^n -1 \geq (1+b^M)\cdot b^1-1 > b^M > q_j^{v_j+N}.
\end{equation*}
Thus, if $\gamma_j < v_j+N$, then $Q>1$ and we are done. On the other hand, if $\gamma_j \geq v_j+N$, then by the same arguments as for $P$, we derive that $Q$ is divisible by the prime $q_j^{*}\neq q_j$. This shows that $k\cdot b^n - 1$ has at least two distinct prime factors, namely $q_j$ and $q_j^{*}$.\par
So far we have shown how for every fixed integer $b\geq3$, to choose such $k$ that both $kb^n+1$ and $kb^n-1$ have at least two distinct prime divisors for every positive integer $n$. For this $k$ to be a $b$-Brier number, we need to ensure that coprimality conditions in Definitions~\ref{sierpn} and~\ref{riesn} hold. In choosing $k$ by the Chinese remainder theorem, besides \eqref{modp}, \eqref{modq}, and \eqref{modb}, we may additionally impose
\[
k\equiv 0 \pmod {b-1}.
\]
This additional condition is justified since all the primes appearing in \eqref{primeslist} are coprime to $b-1$, as the order of $b$ modulo each of them is greater than $1$, while clearly
\[\gcd(b^{M+1},b-1)=1.\]
Consequently
\[
\gcd(k+1,b-1)=\gcd(k-1,b-1)=1.
\]
Together with what was proved above, this shows that $k$ is simultaneously a $b$-Sierpi\'nski number and a $b$-Riesel number, and hence is a $b$-Brier number. Finally, observe that the Chinese remainder theorem gives not just one such integer $k$, but an entire residue class. More precisely, if $k_0>b^M$ is one solution of all the imposed congruences and
\[
L:=(b-1)b^{M+1}
\prod_{i=1}^{C}p_i^{u_i+N}p_i^{*}
\prod_{j=1}^{D}q_j^{v_j+N}q_j^{*},
\]
then every integer
\[
k=k_0+rL,\qquad r\geq0,
\]
satisfies the same congruences. Therefore, the preceding argument applies to every such $k$. Hence
\[
\{k_0+rL:r\geq0\}
\]
is an infinite arithmetic progression of $b$-Brier numbers, where $b\geq3$ is fixed.\par
It remains to prove the theorem for $b=2$. The only obstruction to the preceding argument where we had $b\geq3$ is the modulus $m_2=6$ in $\mathcal{S}_1$, since $2^6-1$ has no primitive prime divisor. For this modulus, we instead associate the prime $p_2=3$. Since neither $\mathcal{S}_1$ nor $\mathcal{S}_2$ contains the modulus $2$, the prime $3$, which is a primitive divisor of $2^2-1$, does not arise from any other modulus, and hence the primes in \eqref{primeslist} remain distinct. Moreover,
\[
\nu_3(2^6-1)=2
\quad\text{and}\quad
\operatorname{ord}_{3^2}(2)=6.
\]
Thus $u_2=2$, and Lemma~\ref{ord} yields
\[
\operatorname{ord}_{3^{2+N}}(2)=6\cdot 3^N.
\]
Therefore, the same order-lifting argument applies to the modulus $m_2=6$, while for all other moduli in $\mathcal{S}_1$ and $\mathcal{S}_2$ the preceding proof remains unchanged. Thus, the theorem also holds for $b=2$.
\end{proof}

{\bf Acknowledgements.} The author would like to thank the anonymous referee for valuable comments and suggestions that substantially improved the clarity and presentation of the manuscript.

\bibliographystyle{plain}
\bibliography{BMBibTeX}

@article {erd50,
    AUTHOR = {Erd\H{o}s, P.},
     TITLE = {On integers of the form {$2^k+p$} and some related problems},
   JOURNAL = {Summa Brasil. Math.},
  FJOURNAL = {Summa Brasiliensis Mathematicae},
    VOLUME = {2},
      YEAR = {1950},
     PAGES = {113--123},
}

@article {b16,
    AUTHOR = {K. Bresz and L. Jones and A. Lamarche and M. Markovich},
     TITLE = {A problem related to a conjecture of {P}olignac},
   JOURNAL = {Integers},
  FJOURNAL = {Integers. Electronic Journal of Combinatorial Number Theory},
    VOLUME = {16},
      YEAR = {2016},
     PAGES = {Paper No. A43, 8 pp.},
}

@article {sierp60,
    AUTHOR = {Sierpi\'nski, W.},
     TITLE = {Sur un probl\`eme concernant les nombres {$k\cdot 2\sp{n}+1$}},
   JOURNAL = {Elem. Math.},
  FJOURNAL = {Elemente der Mathematik. Revue de Math\'ematiques
              \'El\'ementaires. Rivista de Matematica Elementare},
    VOLUME = {15},
      YEAR = {1960},
     PAGES = {73--74},
}

@article {c1,
    AUTHOR = {P. Balister and B. Bollob\'as and R. Morris and J. Sahasrabudhe and M. Tiba},
     TITLE = {Erd{\H{o}}s covering systems},
   JOURNAL = {Acta Math. Hungar.},
  FJOURNAL = {Acta Mathematica Hungarica},
    VOLUME = {161},
      YEAR = {2020},
    NUMBER = {2},
     PAGES = {540--549},
}

@article {c2,
    AUTHOR = {H. Li and B. Wang and C. Wang and S. Yi},
     TITLE = {On {E}rd{\H{o}}s covering systems in global function fields},
   JOURNAL = {J. Number Theory},
  FJOURNAL = {Journal of Number Theory},
    VOLUME = {266},
      YEAR = {2025},
     PAGES = {269--280},
}

@article {c3,
    AUTHOR = {C. Bispels and M. Cohen and J. Harrington and J. Lowrance and K. Pontes and L. Schaumann and T. W. H. Wong},
     TITLE = {A further investigation on covering systems with odd moduli},
   JOURNAL = {Discrete Math.},
  FJOURNAL = {Discrete Mathematics},
    VOLUME = {349},
      YEAR = {2026},
     PAGES = {Paper No. 115013, 14  pp.},
}

@article {gs1,
    AUTHOR = {C. Finch and J. Harrington and L. Jones},
     TITLE = {Nonlinear {S}ierpi\'nski and {R}iesel numbers},
   JOURNAL = {J. Number Theory},
  FJOURNAL = {Journal of Number Theory},
    VOLUME = {133},
      YEAR = {2013},
    NUMBER = {2},
     PAGES = {534--544},
}

@article {gs2,
    AUTHOR = {L. Jones},
     TITLE = {Variations on a theme of {S}ierpi\'nski},
   JOURNAL = {J. Integer Seq.},
  FJOURNAL = {Journal of Integer Sequences},
    VOLUME = {10},
      YEAR = {2007},
    NUMBER = {4},
     PAGES = {Article 07.4.4, 15  pp.},
}

@article {gs3,
    AUTHOR = {M. Filaseta and R. Groth and T. Luckner},
     TITLE = {Generalized {S}ierpi\'nski numbers},
   JOURNAL = {Colloq. Math.},
  FJOURNAL = {Colloquium Mathematicum},
    VOLUME = {174},
      YEAR = {2023},
    NUMBER = {2},
     PAGES = {191--201},
}

@article{riesel56,
  author    = {H. Riesel},
  title     = {N\r{a}gra stora primtal},
  journal   = {Elementa},
  volume    = {39},
  year      = {1956},
  pages     = {258--260},
}

@article {doc09,
    AUTHOR = {C. K. Caldwell and T. Komatsu},
     TITLE = {Powers of {S}ierpi\'nski numbers base {$B$}},
   JOURNAL = {Integers},
  FJOURNAL = {Integers. Electronic Journal of Combinatorial Number Theory},
    VOLUME = {10},
      YEAR = {2010},
        PAGES = {Paper No. A36, 423--436},
}

@article {m1,
    AUTHOR = {Y.-G. Chen},
     TITLE = {On integers of the forms {$k^r-2^n$} and {$k^r2^n+1$}},
   JOURNAL = {J. Number Theory},
  FJOURNAL = {Journal of Number Theory},
    VOLUME = {98},
      YEAR = {2003},
    NUMBER = {2},
     PAGES = {310--319},
}

@article {m3,
    AUTHOR = {Y.-G. Chen},
     TITLE = {On integers of the form {$k2^n+1$}},
   JOURNAL = {Proc. Amer. Math. Soc.},
  FJOURNAL = {Proceedings of the American Mathematical Society},
    VOLUME = {129},
      YEAR = {2001},
    NUMBER = {2},
     PAGES = {355--361},

}

@article {m4,
    AUTHOR = {Y.-G. Chen},
     TITLE = {On integers of the form {$2^k\pm
              p^{\alpha_1}_1p^{\alpha_2}_2\cdots p^{\alpha_r}_r$}},
   JOURNAL = {Proc. Amer. Math. Soc.},
  FJOURNAL = {Proceedings of the American Mathematical Society},
    VOLUME = {128},
      YEAR = {2000},
    NUMBER = {6},
     PAGES = {1613--1616},
}

@article {harr15,
    AUTHOR = {J. Harrington},
     TITLE = {Two questions concerning covering systems},
   JOURNAL = {Int. J. Number Theory},
  FJOURNAL = {International Journal of Number Theory},
    VOLUME = {11},
      YEAR = {2015},
    NUMBER = {6},
     PAGES = {1739--1750},
}

@article {m5,
    AUTHOR = {J.-J. Zhao},
     TITLE = {Integers not of the form {$2^n+p^{\alpha} q^{\beta}$}},
   JOURNAL = {Bull. Korean Math. Soc.},
  FJOURNAL = {Bulletin of the Korean Mathematical Society},
    VOLUME = {61},
      YEAR = {2024},
    NUMBER = {6},
     PAGES = {1631--1637},
}

@article {zsg,
    AUTHOR = {K. Zsigmondy},
     TITLE = {Zur {T}heorie der {P}otenzreste},
   JOURNAL = {Monatsh. Math. Phys.},
  FJOURNAL = {Monatshefte f\"ur Mathematik und Physik},
    VOLUME = {3},
      YEAR = {1892},
    NUMBER = {1},
     PAGES = {265--284},
}

@article {fil08,
    AUTHOR = {M. Filaseta and C. Finch and M. Kozek},
     TITLE = {On powers associated with {S}ierpi\'nski numbers, {R}iesel
              numbers and {P}olignac's conjecture},
   JOURNAL = {J. Number Theory},
  FJOURNAL = {Journal of Number Theory},
    VOLUME = {128},
      YEAR = {2008},
    NUMBER = {7},
     PAGES = {1916--1940},
}

@article {sq1,
    AUTHOR = {A. Armbruster and G. Barger and S. Bykova and
              T. Dvorachek and E. Eckard and J. Harrington and
              Y. Sun and T. W. H. Wong},
     TITLE = {On binomial coefficients associated with {S}ierpi\'nski and
              {R}iesel numbers},
   JOURNAL = {Integers},
  FJOURNAL = {Integers. Electronic Journal of Combinatorial Number Theory},
    VOLUME = {21},
      YEAR = {2021},
     PAGES = {Paper No. A61, 14 pp.},

}

@article {sq2,
    AUTHOR = {D. Baczkowski and J. Eitner},
     TITLE = {Polygonal-{S}ierpi\'nski-{R}iesel sequences with terms having
              at least two distinct prime divisors},
   JOURNAL = {Integers},
  FJOURNAL = {Integers. Electronic Journal of Combinatorial Number Theory},
    VOLUME = {16},
      YEAR = {2016},
     PAGES = {Paper No. A40, 17 pp.},

}

@article {sq3,
    AUTHOR = {F. Luca and V. J. Mej\'ia Huguet},
     TITLE = {Fibonacci-{R}iesel and {F}ibonacci-{S}ierpi\'nski numbers},
   JOURNAL = {Fibonacci Quart.},
  FJOURNAL = {The Fibonacci Quarterly. The Official Journal of the Fibonacci
              Association},
    VOLUME = {46/47},
      YEAR = {2008/09},
    NUMBER = {3},
     PAGES = {216--219},
}

@article {sq4,
    AUTHOR = {D. Ismailescu and Park, P. S.},
     TITLE = {On pairwise intersections of the {F}ibonacci, {S}ierpi\'nski,
              and {R}iesel sequences},
   JOURNAL = {J. Integer Seq.},
  FJOURNAL = {Journal of Integer Sequences},
    VOLUME = {16},
      YEAR = {2013},
    NUMBER = {9},
     PAGES = {Article 13.9.8, 9 pp.},
}

@article {nn,
    AUTHOR = {H. Evans and J. Harrington and K. Heiney and M. Wieczorek},
      title={Generalized {S}ierpi\'{n}ski and {R}iesel numbers of the form $tb^t+\alpha$}, 
      YEAR = {2025},
     PAGES = {13 pp.},
     note = {arXiv:2508.05942}
}

@article {sr26,
    AUTHOR = {C. Bispels and M. Cohen and  J. Harrington and
              J. Lowrance and K. Pontes and L. Schaumann and
              T. W. H. Wong},
     TITLE = {On {S}ierpi\'nski and {R}iesel repdigits and repintegers},
   JOURNAL = {Integers},
  FJOURNAL = {Integers. Electronic Journal of Combinatorial Number Theory},
    VOLUME = {26},
      YEAR = {2026},
     PAGES = {Paper No. A12, 13 pp.},
      ISSN = {1553-1732},
       DOI = {10.5281/zenodo.18154168},
       URL = {https://doi.org/10.5281/zenodo.18154168},
}

@inproceedings{br09,
  author    = {A. Brunner and C. K. Caldwell and D. Krywaruczenko
               and C. Lownsdale},
  title     = {Generalizing {S}ierpi\'{n}ski numbers to base $b$},
  booktitle = {New Aspects of Analytic Number Theory, Proceedings of RIMS (RIMS K{\^o}ky{\^u}roku)},
  number    = {1639},
  pages     = {69--79},
  publisher = {Research Institute for Mathematical Sciences, Kyoto University},
  year      = {2009}
}

@article {cg07,
    AUTHOR = {Y.-G. Chen},
     TITLE = {On integers of the forms {$k\pm 2^n$} and {$k2^n\pm 1$}},
   JOURNAL = {J. Number Theory},
  FJOURNAL = {Journal of Number Theory},
    VOLUME = {125},
      YEAR = {2007},
    NUMBER = {1},
     PAGES = {14--25},
      ISSN = {0022-314X,1096-1658},
   MRCLASS = {11B25 (11A07 11P32)},
MRREVIEWER = {Christian\ Elsholtz},
       DOI = {10.1016/j.jnt.2006.10.005},
       URL = {https://doi.org/10.1016/j.jnt.2006.10.005},
}

\end{document}